\documentclass[reqno,b5paper]{amsart}
\usepackage{amsmath,amssymb,amsbsy,amsfonts,amsthm,latexsym,amsopn,amstext,amsxtra,euscript,amscd,amsthm}

\newtheorem{lemma}{Lemma}
\newtheorem{prop}{Proposition}
\newtheorem{theorem}{Theorem}

\newcommand{\Z}{\mathbb Z}
\newcommand{\Q}{\mathbb Q}
\newcommand{\C}{\mathbb C}
\newcommand{\ve}{\varepsilon}
\newcommand{\ka}{\epsilon}
\newcommand{\cc}{\kappa}

\newcommand{\acr}{\newline\indent}

\begin{document}

\title{Only finitely many Tribonacci Diophantine triples exist}

\author[C. Fuchs \and C. Hutle \and N. Irmak \and F. Luca \and L. Szalay]{Clemens Fuchs* \and Christoph Hutle* \and Nurettin Irmak** \and Florian Luca*** \and Laszlo Szalay****}

\address{\llap{*\,}University of Salzburg \acr Hellbrunner Str. 34/I \acr 5020 Salzburg \acr AUSTRIA}
\email{clemens.fuchs@sbg.ac.at, christoph.hutle@gmx.at}
\address{\llap{**\,}U.~Ni\c{g}de, TURKEY}
\email{irmaknurettin@gmail.com}
\address{\llap{***\,}U.~Witwatersrand, SOUTH AFRICA}
\email{florian.luca@wits.ac.za}
\address{\llap{****\,}U.~West Hungary, HUNGARY}
\email{szalay.laszlo@emk.nyme.hu}

\thanks{C.F. and C.H. were supported by FWF (Austrian Science Fund) grant No. P24574 and by the Sparkling Science project EMMA grant No. SPA 05/172.}

\maketitle

\begin{abstract}
Diophantine triples taking values in recurrence sequences have recently been studied quite a lot. In particular the question was raised whether or not there are finitely many Diophantine triples in the Tribonacci sequence. We answer this question here in the affirmative. We prove that there are only finitely many triples of integers $1\le u<v<w$ such that $uv+1,uw+1,vw+1$ are Tribonacci numbers. The proof depends on the Subspace theorem.
\end{abstract}

\section{Introduction}\label{sec:Intro}

The theory of diophantine tuples has a long history and origins in work by Diophantus. The basic question is to construct sets of rationals or integers with the property that the product of any two of its distinct elements plus $1$ is a square. Many facts are known on this problem most notably Dujella's result (cf. \cite{duje0}) that there are at most finitely many sets of five positive integers with this property (that the product of any two distinct elements from the set plus $1$ is a square); further results and the attractive history of the problem can be found in \cite{duje}. Besides the main streamline several variations of this problem have been investigated.

The problem of finding bounds on the size $m$ for Diophantine $m$-tuples with values in linear recurrences is one such variation. The first general result
is due to Fuchs, Luca and Szalay \cite{FLSz}  and states that if $\{u_n\}_{n\ge1}$ is a binary recurrent sequence satisfying certain conditions, then there are at most finitely many triples
of positive integers $a < b < c$ such that $ab+1,$ $ac+1$ and $bc+1$ are all members of $\{u_n\}_{n\ge1}$.
The sequences of Fibonacci and Lucas numbers satisfy the conditions of the above theorem, and all Diophantine triples with values in the Fibonacci sequence
or in the Lucas sequence were computed in \cite{LSz1} and \cite{LSz2}, respectively. Later, in \cite{ISz1}, the method was extended to find all Diophantine triples with values
in a certain parametric family of Lucas sequences which includes the balancing sequence as a particular case.

All the previous works dealt with binary recurrences.
Concerning linear recurrences of higher order we mention \cite{RL} and \cite{RL1} in which it is shown that there are no Diophantine quadruples with values in the Tribonacci sequence $\{T_n\}_{n\ge 1}$ or in the shifted Tribonacci sequence $\{T_n+1\}_{n\ge 1}$.
Recall that this sequence has the property that $T_{0}=T_1=0,$ $T_2=1$ and $T_{n+3}=T_{n+2}+T_{n+1}+T_n$ for all $n\ge 0$. In \cite{RL}, they conjectured that only finitely many Diophantine triples  exist with values in the Tribonacci sequence, but by their method it was not possible to deal with this problem.

In this paper, we prove the above conjecture. More precisely, we have the following theorem.
\begin{theorem}\label{thm:main}
There are only finitely triples of integers $1\le u<v<w$ such that
\begin{equation}\label{eq:triples}
1+uv=T_x,\quad 1+uw=T_y,\quad 1+vw=T_z
\end{equation}
hold for some positive integers $x$, $y$, $z$.
\end{theorem}

First we recall some facts concerning the Tribonacci sequence $\{T_n\}_{n\ge 0}$.
The characteristic equation
$$
x^3-x^2-x-1=0
$$
has roots $\alpha,~\beta,~\gamma={\overline{\beta}}$, where
$$
\alpha=\frac{1+\omega_1+\omega_2}{3},\qquad \beta=\frac{2-\omega_1-\omega_2+{\sqrt{3}} i (\omega_1-\omega_2)}{6},
$$
and
$$
\omega_1=\sqrt[3]{19+3{\sqrt{33}}}\quad {\text{\textrm{and}}}\quad \omega_2=\sqrt[3]{19-3{\sqrt{33}}}.
$$
Further, Binet's formula is
\begin{equation}
\label{eq:Binet}
T_n=a\alpha^n+b\beta^n+c\gamma^n\quad {\text{\textrm{for all}}}\quad n\ge 0,
\end{equation}
where
\begin{equation}
\label{eq:coefficients}
a=\frac{1}{(\alpha-\beta)(\alpha-\gamma)},\quad b=\frac{1}{(\beta-\alpha)(\beta-\gamma)},\quad c=\frac{1}{(\gamma-\alpha)(\gamma-\beta)}={\overline{b}}
\end{equation}
(see \cite{S}). Numerically,
\begin{equation}
\label{eq:num}
\begin{split}
& 1.83<\alpha<1.84,\\
& 0.73<|\beta|=|\gamma|=\alpha^{-1/2}<0.74,\\
& 0.18<a<0.19,\\
& 0.35<|b|=|c|<0.36.
\end{split}
\end{equation}
Further,
\begin{equation}
\label{eq:BL}
\alpha^{n-3}\le T_n\le \alpha^{n-2}
\end{equation}
for all $n\ge 2$ (see \cite{BL}).

In Section \ref{sec:gcd}, we give an upper bound for the greatest common divisor (gcd) of $T_y-1$ and $T_z-1$. This result plays an important role in the proof of our Theorem \ref{thm:main}. Section \ref{sec:proof} proves Theorem \ref{thm:main} apart from Lemmas \ref{lem:1} and \ref{lem:2}, which are proved in Section \ref{sec:lemmas}. The main tools in the proof are the Subspace theorem (we use the version from \cite{Ev}) and the theory of $S$-unit equations (cf. \cite{ESS}), which we use to prove Lemma \ref{lem:1}.

\section{On the gcd of $T_y-1$ and $T_z-1$}
\label{sec:gcd}

We start with the following:

\begin{prop}
\label{prop:1}
If $4\le y<z$, then
$$
\gcd(T_y-1,T_z-1)<\alpha^{3z/4}.
$$
\end{prop}

\begin{proof}
We may assume that $y\ge 5$, otherwise $y=4$, $T_4-1=1$, and there is nothing to prove.
Let $\cc\in (0,1)$ be some constant to be determined later. Let $d=\gcd(T_y-1,T_z-1)$. If $y\le \cc z+2$, then
\begin{equation}
\label{eq:c111}
d=\gcd(T_y-1,T_z-1)\le T_y-1<T_y\le \alpha^{y-2}\le \alpha^{\cc z}.
\end{equation}
From now on, we assume that $y>\cc z+2$. Let $\lambda=z-y<(1-\cc)z$. Then
\begin{equation*}\begin{split}
a \alpha^{y+\lambda}+b \beta^{y+\lambda}+c \gamma^{y+\lambda}-1 & \equiv   0\pmod d\\
a\alpha^y +b\beta^y +c\gamma^y-1 & \equiv  0\pmod d.
\end{split}\end{equation*}
Multiplying the second congruence above by $\alpha^{\lambda}$ and subtracting from the result the first congruence above gives
$$
b\beta^y (\alpha^{\lambda}-\beta^{\lambda})+c\gamma^y(\alpha^{\lambda}-\gamma^{\lambda})-(\alpha^{\lambda}-1)\equiv 0\pmod d.
$$
The left-hand side above is just $\alpha^{\lambda}(T_y-1)-(T_z-1)$. It is an algebraic integer which is not zero because otherwise we get that $\alpha^{\lambda}=(T_z-1)/(T_y-1)\in {\mathbb Q}$, which is false for $\lambda>0$. Thus,
\begin{equation}
\label{eq:1}
d\eta=b\beta^y (\alpha^{\lambda}-\beta^{\lambda})+c\gamma^y(\alpha^{\lambda}-\gamma^{\lambda})-(\alpha^{\lambda}-1)
\end{equation}
holds with some nonzero algebraic integer $\eta\in {K}:= {\mathbb Q}(\alpha,\beta)$.
The Galois group $G={\text{\textrm{Gal}}}({K}/{\mathbb Q})$ is $S_3$. Taking norms from ${K}$ to ${\mathbb Q}$ in \eqref{eq:1} gives
\begin{equation}
\label{eq:2}
\begin{split}
d^6  & \le   d^6|N_{K/{\mathbb Q}}(\eta)|=|N_{K/{\mathbb Q}}(d\eta)|\\
& =  \prod_{\sigma \in G} \big|\sigma(b) \sigma(\beta)^y(\sigma(\alpha)^{\lambda}-\sigma(\beta)^{\lambda})\\
& \hspace*{2cm} +\sigma(c) \sigma(\gamma)^y(\sigma(\alpha)^{\lambda}-\sigma(\gamma)^{\lambda})-(\sigma(\alpha)^{\lambda}-1)\big|.
\end{split}
\end{equation}
Now $G$ has two elements fixing $\alpha$ (the identity and the involution $(\beta$, $\gamma)$). The other four elements of $G$ map $\alpha$ into one of $\beta$, $\gamma$ and one of $\beta$, $\gamma$ into $\alpha$. Also, any element of $G$ induces a permutation of the coefficients $a,b,c$ (by \eqref{eq:coefficients}) in the Binet formula \eqref{eq:Binet} and $|b|=|c|>a$. We study the size of the factors in the product in the right-hand side of \eqref{eq:2}:
\begin{itemize}
\item[(i)] If $\sigma(\alpha)=\alpha$, then, by the absolute value inequality,
\begin{equation*}\begin{split}
&\left|\sigma(b) \sigma(\beta)^y(\sigma(\alpha)^{\lambda}-\sigma(\beta)^{\lambda})+\sigma(c) \sigma(\gamma)^y(\sigma(\alpha)^{\lambda}-\sigma(\gamma)^{\lambda})-(\sigma(\alpha)^{\lambda}-1)\right|\\
& \le  \alpha^{\lambda} (2|b| |\beta|^y)+2|b||\beta|^z+\alpha^{\lambda}-1\\
& =  \alpha^{\lambda}\left(1+\frac{2|b|}{|\alpha|^{y/2}}\right)+\left(\frac{2|b|}{|\alpha|^{z/2}}-1\right)<1.3\alpha^{\lambda}<1.3 \alpha^{(1-\cc)z}.
\end{split}\end{equation*}
Here, we used the fact that $y\ge 5$, so $z\ge 6$, and the numerics \eqref{eq:num}.
\item[(ii)] If $\sigma(\alpha)\ne \alpha$, it then follows that, by applying again the absolute value inequality,
\begin{equation*}\begin{split}
&\left|\sigma(b) \sigma(\beta)^y(\sigma(\alpha)^{\lambda}-\sigma(\beta)^{\lambda})+\sigma(c) \sigma(\gamma)^y(\sigma(\alpha)^{\lambda}-\sigma(\gamma)^{\lambda})-(\sigma(\alpha)^{\lambda}-1)\right|\\
& =  |\sigma(\alpha)^{\lambda}(T_y-1)-(T_z-1)|\le  T_z-1+|\alpha|^{-\lambda/2} (T_y-1)\\
& <  2T_z<(2\alpha^{-2}) \alpha^z<0.6 \alpha^z.
\end{split}\end{equation*}
Here, we used \eqref{eq:BL}.
\end{itemize}
With inequality \eqref{eq:2}, we get
$$
d^6<(1.3 \alpha^{(1-\cc)z})^2 (0.6 \alpha^z)^4<\alpha^{(6-2\cc)z},
$$
giving
\begin{equation}\label{eq:c2}
d<\alpha^{(1-\cc/3)z}.
\end{equation}
Taking $\cc=3/4$ to balance between \eqref{eq:c111} and \eqref{eq:c2}, we get $d<\alpha^{3z/4}$, which is what we wanted to prove.
\end{proof}

\section{Proof of Theorem 1}
\label{sec:proof}

We shall assume that there are infinitely many such triples $(u,v,w)$ with corresponding $(x,y,z)$ and eventually reach a contradiction. Solving for $u,v,w$ in terms of $x,y,z$ from \eqref{eq:triples}, we get
\begin{equation}
\label{eq:uvw}
u=\sqrt{\frac{(T_x-1)(T_y-1)}{T_z-1}}, v=\sqrt{\frac{(T_x-1)(T_z-1)}{T_y-1}}, w=\sqrt{\frac{(T_y-1)(T_z-1)}{T_x-1}}.
\end{equation}
Since $1\le u<v<w$, we get $5\le x<y<z$. Further, the case $z=7$ entails $(x,y)=(5,6)$, but then the corresponding $(u,v,w)$ in \eqref{eq:uvw} are not integers. Thus, $z\ge 8$.
Since $u\ge 1$, we get, by \eqref{eq:BL},
$$
\alpha^{x+y-4} \ge T_x T_y>(T_x-1)(T_y-1)\ge T_z-1\ge \alpha^{z-3}-1>\alpha^{z-4},
$$
giving $x+y>z$. Hence, $y>z/2$. Further,
$$
T_z-1\mid (T_y-1)(T_x-1).
$$
Let $d_1=\gcd(T_z-1,T_y-1)$ and $d_2=\gcd(T_z-1,T_x-1)$. So, $T_z-1\mid d_1 d_2$. Hence, by \eqref{eq:BL} and Proposition \ref{prop:1}, we have
$$
\alpha^{x-2} \ge T_x>T_x-1\ge d_2\ge \frac{T_z-1}{d_1}\ge \frac{\alpha^{z-3}-1}{\alpha^{3z/4}}>\alpha^{z/4-4},
$$
giving
\begin{equation}
\label{eq:xversusz}
x\ge z/4-2\ge z/12,
\end{equation}
at least for $z\ge 12$. Since there can be only finitely many such triples $(x,y,z)$ with $z<12$, we can assume $x\ge z/12$ for all the triples that we consider from now on.

We now use results from Diophantine approximations. Put
$$
a_1:=-1/a,\quad b_1:=b/a,\quad c_1:=c/a,
$$
and then
\begin{equation}\label{eq:c}\begin{split}
u & =  {\sqrt{a}} \alpha^{(x+y-z)/2} \left(1+a_1 \alpha^{-x}+b_1\beta^x \alpha^{-x}+c_1\gamma^x \alpha^{-x}\right)^{1/2}\nonumber\\
& \hspace{1cm}\cdot  \left(1+a_1\alpha^{-y}+b_1\beta^{y}\alpha^{-y}+c_1\gamma^{y}\alpha^{-y}\right)^{1/2}\nonumber\\
& \hspace{1cm}\cdot  \left(1+a_1\alpha^{-z}+b_1\beta^z \alpha^{-z}+c_1\gamma^{z}\alpha^{-z}\right)^{-1/2}\nonumber\\
& =  {\sqrt{a}} \alpha^{(x+y-z)/2} \left(1+\sum_{i\ge 1} d_{u,i} \alpha^{A_{u,i}({\bf x})} \beta^{B_{u,i}({\bf x})} \gamma^{C_{u,i}({\bf x})}\right),
\end{split}\end{equation}
where  ${\bf x}=(x,y,z)$ and for each $i\ge 1$ the numbers $d_{u,i}$ are some coefficients which are the product between some rational number (actually, a $2$-unit) and a monomial (with nonnegative degrees in each indeterminate) in $a_1,b_1,c_1$. Here we are taking the real positive root of $a$ and $\alpha$ respectively. The expansion is obtained as follows: We use the binomial series to get \begin{equation*}\begin{split}(1+&a_1\alpha^{-x}+b_1\beta^x\alpha^{-x}+c_1\gamma^x\alpha^{-x})^{1/2}\\ &=\sum_{k=0}^T{1/2\choose k}\left(a_1\alpha^{-x}+b_1\beta^x\alpha^{-x}+c_1\gamma^x\alpha^{-x}\right)^k+O\left(\alpha^{-(T+1)x}\right),\end{split}\end{equation*} where $O$ has the usual meaning (for the error term see e.g. \cite[Lemma 2]{ft}). Similarly we expand the terms $\left(1+a_1\alpha^{-y}+b_1\beta^{y}\alpha^{-y}+c_1\gamma^{y}\alpha^{-y}\right)^{1/2}$ and $\left(1+a_1\alpha^{-z}\right.$ $\left.+b_1\beta^z \alpha^{-z}+c_1\gamma^{z}\alpha^{-z}\right)^{-1/2}$ up to index $T$. By multiplying out the expressions we indeed get $$u=\sqrt{a}\alpha^{(x+y-z)/2}\left[1+\sum_{i=1}^{n-1}d_{u,i}\alpha^{A_{u,i}({\bf x})} \beta^{B_{u,i}({\bf x})} \gamma^{C_{u,i}({\bf x})}\right]+O(\alpha^{-Tx}),$$ where the integer $n$ depends only on $T$ and the terms are ordered in some arbitrary way. Since $x\ge z/12$ and $x<y<z$ we obtain \begin{equation}\label{err} u=\sqrt{a}\alpha^{(x+y-z)/2}\left[1+\sum_{i=1}^{n-1}d_{u,i}\alpha^{A_{u,i}({\bf x})} \beta^{B_{u,i}({\bf x})} \gamma^{C_{u,i}({\bf x})}\right]+O(\alpha^{-T\Vert {\bf x}\Vert /12}),\end{equation} where we have put $\Vert {\bf x}\Vert=z=\max\{x,y,z\}$. We shall use this bound later on.

Observe that $A_{u,i}({\bf x})$, $B_{u,i}({\bf x})$ and $C_{u,i}({\bf x})$ are linear forms in ${\bf x}$ with integer coefficients in such a way that the coefficients of $A_{u,i}({\bf x})$ are always $\le 0$ and the coefficients of $B_{u.i}({\bf x})$ and $C_{u,i}({\bf x})$ are always $\ge 0$.

We would like to use the Subspace Theorem (see e.g. \cite{Ev}; the version we are going to use can also be found in Section 3 of \cite{Fu2}, which's notation - in particular the notion of heights - we will also use) in order to show, that the expansion for $u$ is ``essentially'' finite (compare with \cite{Fu2,Fu5,Fu4}; however, observe that the results from these sources are not directly applicable and that we have to work through the estimates instead).

For this, we again take $K := \Q(\alpha, \beta)$ and choose for $S$ the set of all places (normalized as usual so that the Product Formula holds; see \cite{Ev,Fu2}), which are either infinite or in the set $\{ v \in M_K: \vert \alpha\vert_v \neq 1 \vee \vert \beta\vert_v \neq 1 \vee \vert \gamma\vert_v \neq 1\}$.

Using the fixed $n$ (depending on $T$) from above, we will now define $n+1$ linearly independent sets of linear forms with indeterminants $(U, Y_0, \dots, Y_{n-1})$ for each $v \in S$.

We have to distinguish two cases according to whether $x+y-z$ is even or odd. In other words we have either $x+y-z=2k$ or $x+y-z=2k+1$, respectively, which can be rewritten as $k=(x+y-z)/2$ resp. $k=(x+y-z-1)/2$ with $k\in \Z$. We put $\ka\in \{0,1\}$ to distinguish between these cases, whereby we put $\ka=0$ if $x+y-z=2k$ and $\ka=1$ if $x+y-z=2k+1$. By going to a still infinite subset of the solutions we may assume that $\ka$ is fixed.

For the standard infinite place $\infty$ on $\C$, we define
\begin{equation*}
l_{0, \infty}(U, Y_0, \dots, Y_{n-1}) := U - \sqrt{a\alpha^\ka} Y_0 - \sqrt{a\alpha^\ka} \sum_{i=1}^{n-1} d_{u,i} Y_i
\end{equation*}
and
$$
l_{i, \infty} := Y_{i-1} \quad \textrm{for } i \in \{1, \dots, n\}.
$$
For all other places $v$ in $S$, we define
$$
l_{0, v} := U, \qquad l_{i, v} := Y_{i-1} \quad \textrm{ for } i = 1, \dots, n.
$$
We will show, that there is some $\delta > 0$, such that the inequality
\begin{equation}\label{eq-sst}
\prod_{v\in S}\prod_{i=0}^{n}\frac{\vert l_{i,v}({\bf
y})\vert_{v}}{\vert {\bf y}\vert_{v}} <\left(\prod_{v\in S}\vert
\det(l_{0,v},\ldots,l_{n,v})\vert_{v}\right)\cdot \mathcal{H}({\bf
y})^{-n-1-\delta}
\end{equation}
is satisfied for all vectors
\begin{equation}\label{eq:sstvec}
\begin{split}
{\bf y}=(&u, \alpha^{(x+y-z-\ka)/2}, \alpha^{(x+y-z-\ka)/2} \cdot \alpha^{A_{u,1}(\bf x)} \beta^{B_{u,1}(\bf x)} \gamma^{C_{u,1}(\bf x)}, \\
&\ldots,\alpha^{(x+y-z-\ka)/2} \cdot \alpha^{A_{u,n-1}(\bf x)} \beta^{B_{u,n-1}(\bf x)} \gamma^{C_{u,n-1}(\bf x)}).
\end{split}
\end{equation}
Observe that it is here where we need the case distinction ($\ka=0$ resp. $\ka=1$) in order to assure that the vector ${\bf y}$ is in $K^{n+1}$.
Since by our choice of linear forms $\det(l_{0,v},\ldots,l_{n,v}) = 1$ for all places in $S$ and
$$
\prod_{v \in S}\prod_{i=0}^n \frac{1}{\vert {\bf y} \vert_v} = \mathcal{H}({\bf y})^{-n-1},
$$
this inequality reduces to
\begin{equation}\label{eq:sstnew}
\prod_{v\in S}\prod_{i=0}^{n} \vert l_{i,v}({\bf y})\vert_{v} < \mathcal{H}({\bf y})^{-\delta}.
\end{equation}
We split up the double product on the left-hand side into
$$
\left\vert u - \sqrt{a\alpha^\ka} y_0 - \sqrt{a\alpha^\ka} \sum_{i=1}^{n-1} d_{u,i} y_i \right\vert_\infty \cdot \prod_{\substack{v \in M_{K,\infty} \\ v \neq \infty}} \vert u \vert_v \cdot \prod_{v \in S \backslash M_{K,\infty}} \vert u \vert_v \cdot \prod_{i=0}^{n-1} \prod_{v \in S} \vert y_i \vert.
$$
Inserting the vector (\ref{eq:sstvec}), we see that $\prod_{v \in S \backslash M_{K,\infty}} \vert u \vert_v \le 1$, since $u$ is some integer. Furthermore, the last double product equals $1$ due to the Product Formula.
Furthermore, we estimate
\begin{equation*}
\begin{split}
\prod_{\substack{v \in M_{K,\infty} \\ v \neq \infty}} \vert u \vert_v &< \left( \frac{(T_x - 1)(T_y - 1)}{T_z - 1}\right)^3 \\
& \le  (a \alpha^x + b \beta^x + c \gamma^x - 1)^3 (a \alpha^y + b \beta^y + c \gamma^y - 1)^3 \\
&\le  \left( 4 \cdot 4^{\| {\bf x} \|} \right)^3.
\end{split}
\end{equation*}
And finally, the first expression is just
\begin{equation*}
\left\vert \sqrt{a\alpha^\ka}\alpha^{(x+y-z-\ka)/2} \sum_{i\geq n} d_{u,i} \alpha^{A_{u,i}(\bf x)} \beta^{B_{u,i}(\bf x)} \gamma^{C_{u,i}(\bf x)} \right\vert_\infty,
\end{equation*}
which, by (\ref{err}), is smaller than $C_1\alpha^{-T\| {\bf x}\|/12}$.
Thus
\begin{equation}\label{eq:lhs1}
\prod_{v\in S}\prod_{i=0}^{n} \vert l_{i,v}({\bf y})\vert_{v} < C_1 \cdot \alpha^{-T\|{\bf x}\|/12} \cdot \left( 4 \cdot 4^{\| {\bf x} \|} \right)^3.
\end{equation}
Now we choose $T$ (and the corresponding $n=n(T)$) large enough, such that
\begin{equation*}\begin{split}
\alpha^{-T/12} \cdot 64 &< \alpha^{-T/13}, \\
64 C_1 &< \alpha^{T/26}.
\end{split}\end{equation*}
Then (\ref{eq:lhs1}) implies
\begin{equation}\label{eq:lhs}
\prod_{v\in S}\prod_{i=0}^{n} \vert l_{i,v}({\bf y})\vert_{v} < \alpha^{-\frac{T}{26} \| {\bf x} \|}.
\end{equation}

Inserting the vector (\ref{eq:sstvec}) on the right-hand side of (\ref{eq:sstnew}), we find
\begin{equation*}\begin{split}
\mathcal{H}( {\bf y}) &\le  C_2 \cdot \mathcal{H}(u) \cdot \mathcal{H}(\alpha^{(x+y-z-\ka)/2})^n \cdot \prod_{i=1}^{n-1} \mathcal{H}(\alpha^{A_{u,i}(\bf x)} \beta^{B_{u,i}(\bf x)} \gamma^{C_{u,i}(\bf x)}) \\
&\le C_2 \cdot (4 \cdot 4^{\| {\bf x} \|}) \cdot \prod_{i=1}^{n-1} \alpha^{C_3 \| {\bf x}\|} \\
&\le \alpha^{C_4 \| {\bf x} \|},
\end{split}\end{equation*}
using suitable constants. For the second inequality, we used that $\mathcal{H}(u)$ is just the product $\prod_{v \in M_{K,\infty}} \vert u \vert_v$ (since $u$ is an integer), which allows a similar estimation as above, and also that
\begin{equation*}\begin{split}
\mathcal{H}(\alpha^{(x+y-z-\ka)/2})^n &\le \alpha^{n \|{\bf x} \|}, \\
\mathcal{H}(\alpha^{A_{u,i}(\bf x)} \beta^{B_{u,i}(\bf x)} \gamma^{C_{u,i}(\bf x)}) &\le \mathcal{H}(\alpha)^{C_\alpha \| {\bf x} \|} \mathcal{H}(\beta)^{C_\beta \| {\bf x} \|} \mathcal{H}(\gamma)^{C_\gamma \| {\bf x} \|},
\end{split}\end{equation*}
by using the maximum of those finitely many expressions for $i = 1, \dots, n$.
Therefore
\begin{equation}\label{eq:rhs}
\mathcal{H}( {\bf y})^{-\delta} \ge \alpha^{-\delta C_4 \| {\bf x} \|}.
\end{equation}

Now, because of (\ref{eq:lhs}) and (\ref{eq:rhs}) it just remains to pick $\delta$, such that
\begin{equation*}
\alpha^{-\frac{T}{26} \| {\bf x} \|} \le \alpha^{-\delta C_4 \| {\bf x} \|}
\end{equation*}
holds for all of our infinitely many solutions ${\bf x}$. This works, if we choose $\delta \le T/(26 C_4)$.

Thus we can apply the Subspace Theorem, which says, that all solutions $(x,y,z)$ of (\ref{eq-sst}) lie in finitely many proper linear subspaces. In particular, there are still infinitely many solutions in one of these proper subspaces, giving a finite set $I_u$ and (new) coefficients $e_u$ and $e_{u,i}$ for $i\in I_u$ in $K$ such that
\begin{equation}
\label{eq:u}
u=\alpha^{(x+y-z-\ka)/2} \left(e_u+\sum_{i\in I_u} e_{u,i} \alpha^{A_{u,i}({\bf x})} \beta^{B_{u,i}({\bf x})} \gamma^{C_{u,i}({\bf x})}\right)
\end{equation}
holds for still infinitely many $(x,y,z)$; here we eventually have to go to a still infinite subset.
A similar argument holds for the other two variables $v$ and $w$ and we get formulas
\begin{equation}\label{eq:vandw}\begin{split}
v & = \alpha^{(x-y+z-\ka)/2} \left(e_v+\sum_{i\in I_v} e_{v,i} \alpha^{A_{v,i}({\bf x})} \beta^{B_{v,i}({\bf x})} \gamma^{C_{v,i}({\bf x})}\right),\\
w & = \alpha^{(-x+y+z-\ka)/2} \left(e_w+\sum_{i\in I_w} e_{w,i} \alpha^{A_{w,i}({\bf x})} \beta^{B_{w,i}({\bf x})} \gamma^{C_{w,i}({\bf x})}\right),
\end{split}\end{equation}
which also holds for appropriate choice of the coefficients for infinitely many of the $(x,y,z)$.
So, we have arrived at the situation when there exist finite sets of indices $I_u,~I_v,~I_w$ and triples of linear forms $(A_{u,i},B_{u,i},C_{u,i})$ for $i\in I_u$,
$(A_{v,i},B_{v,u},C_{v,i})$ for $i\in I_v$ and $(A_{w.i},B_{w,i},C_{w,i})$ for $i\in I_w$, such that for infinitely many $(x,y,z)$ the quantities $u,v,w$ given by $\eqref{eq:u}$ and \eqref{eq:vandw}
fulfill \eqref{eq:triples}.

\begin{lemma}
\label{lem:1}
Suppose that (\ref{thm:main}) has infinitely many solutions. Then there exists a line in ${\mathbb R}^3$ given by
$$
x(t)=r_1 t+s_1,\quad y(t)=r_2 t+s_2,\quad z(t)=r_3 t+s_3
$$
with rationals $(r_1,r_2,r_3,s_1,s_2,s_3)$ such that infinitely many of the solutions $(x,y,z)$ to (\ref{thm:main}) are of the form  $(x(n),y(n),z(n))$ for some integer $n$.
\end{lemma}

We shall give the prove of this lemma in Section \ref{sec:lemmas}.

Assume we have Lemma \ref{lem:1}. Then there are infinitely many $n$ such that $(x(n),y(n),z(n))$ satisfies our equations. Let $\Delta$ be the least common multiple of the denominators of $r_1,r_2,r_3$.
Infinitely many of our $n$ will be in the same residue class modulo $\Delta$, let us call it $r$. Then writing $n=m\Delta+r$, we get
$$
(x,y,z)=((r_1\Delta)m+(rr_1+s_1), (r_2\Delta) m+(rr_2+s_2),(r_3\Delta)m+(rr_3+s_3)).
$$
Since $r_i\Delta$ are integers for $i=1,2,3$, and $x,y,z$ are integers, we get that $r r_i+s_i$ are integers for $i=1,2,3$. Thus, replacing $n$ by $m$, $r_i$ by $r_i \Delta$ and $s_i$ by $r r_i+s_i$ for $i=1,2,3$, we may assume that $r_i$ and $s_i$ are all integers for $i=1,2,3$. Now
$$
\frac{x+y-z-\ka}{2}=\frac{(r_1+r_2-r_3)m}{2}+\frac{s_1+s_2-s_3-\ka}{2}.
$$
Since we have infinitely many $m$, it follows that infinitely many of them will be in the same residue class modulo $2$. Say this residue class is $\delta\in \{0,1\}$, and write $m=2\ell+\delta$. Then
$$
\frac{x+y-z-\ka}{2}=(r_1+r_2-r_3)\ell+S,
$$
where $S\in {\mathbb Z}$ or ${\mathbb Z}+1/2$. So, now
\begin{equation}
\label{eq:c1}
u(\ell)= \alpha^{(r_1+r_2-r_3)\ell+S} \left(e_u+\sum_{i\in I_u} e_{u,i} \alpha^{A_{u,i}({\bf x})} \beta^{B_{u,i}({\bf x})} \gamma^{C_{u,i}({\bf x})}\right)
\end{equation}
is an integer for infinitely many $\ell$, where here
$$
{\bf x}={\bf x}(\ell)=(x(2\ell+\delta),y(2\ell+\delta),z(2\ell+\delta)).
$$
From this, we will now derive a contradiction.

First we observe, that there are only finitely many solutions of (\ref{eq:c1}) with $u(\ell) = 0$. That can be shown by using the fact, that a simple non-degenerate linear recurrence has only finite zero-multiplicity (see \cite{ESS} for an explicit bound). We will apply this statement here for the linear recurrence in $\ell$; it only remains to check, that no quotient of two distinct roots of the form $\alpha^{A_{u,i}({\bf x(\ell)})} \beta^{B_{u,i}({\bf x(\ell)})}$ $\gamma^{C_{u,i}({\bf x(\ell)})}$ is  a root of unity or, in other words, that \begin{equation}\label{eq:rootunity1}
(\alpha^{m_\alpha} \beta^{m_\beta} \gamma^{m_\gamma})^n = 1
\end{equation}
has no solutions in $n \in \Z / \{0\}$, $m_\alpha < 0$ and $m_\beta, m_\gamma > 0$.
Assume relation \eqref{eq:rootunity1} holds. Suppose $n$ is even (if not replace \eqref{eq:rootunity1} by its square). Then we apply the complex conjugation automorphism, that invaries $\alpha$ and switches $\beta$ and $\gamma$ and we get
\begin{equation}
\label{eq:rootunity2}
(\alpha^{m_{\alpha}} \beta^{m_{\gamma}} \gamma^{m_{\beta}})^n=1.
\end{equation}
Multiplying \eqref{eq:rootunity1} and \eqref{eq:rootunity2} we get
$$
(\alpha^{2m_{\alpha}} (\beta\gamma)^{m_{\beta}+m_{\gamma}})^n=1.
$$
Since $\beta\gamma=-\alpha^{-1}$, and $n$ is even, we obtain
$$
\alpha^{(2m_{\alpha}-m_{\beta}-m_{\gamma})n}=1,
$$
and $2m_{\alpha}-m_{\beta}-m_{\gamma}<0$, a contradiction.

So now, we can assume, that $u(\ell) \neq 0$ for still infinitely many solutions. Using (\ref{eq:uvw}), we can write \begin{equation}\label{eq:expansions} (T_z - 1)u^2 = (T_x - 1)(T_y - 1) \end{equation} and insert the finite expansion (\ref{eq:c1}) in $\ell$ for $u$ into (\ref{eq:expansions}). Furthermore, we use the Binet formula (\ref{eq:Binet}) and write $T_x, T_y$ and $T_z$ as power sums in $x, y$, and $z$ respectively. Using the parametrization $(x,y,z) = (r_1 m + s_1, r_2 m + s_2, r_3 m + s_3)$ with $m = 2 \ell$ or $m = 2\ell + 1$ as above, we have expansions in $\ell$ on both sides of (\ref{eq:expansions}). Since there must be infinitely many solutions in $\ell$, the largest terms on both sides have to grow with the same rate. In order to find the largest terms, we have to distinguish some cases: If we assume, that $e_u \neq 0$ for infinitely many of our solutions, then $e_u \alpha^{(x+y-z-\ka)/2}$ is the largest term in the expansion of $u$ and we have $$ a\alpha^ze_u^2\alpha^{x+y-z-\ka}=a\alpha^xa\alpha^y.$$ It follows that $e_u^2=a\alpha^\ka$.
The case $e_u=0$ for infinitely many of our solutions is not possible, because the right-hand side of (\ref{eq:expansions}) would grow faster than the left-hand side so that (\ref{eq:expansions}) would be true for at most finitely many of our $\ell$.
In all other cases, we had $e_u = \sqrt{a\alpha^\ka}$, where $\ka\in\{0,1\}$. But this contradicts the following lemma.

\begin{lemma}
\label{lem:2}
${\sqrt{a}}\notin{K}$ and ${\sqrt{a \alpha}}\notin K$.
\end{lemma}
We will prove this lemma also in Section \ref{sec:lemmas}.
The contradiction proves our theorem.
\qed

\section{The proofs of Lemmas \ref{lem:1} and \ref{lem:2}}
\label{sec:lemmas}

We start with Lemma \ref{lem:2}.

\medskip
\noindent{\it Proof of Lemma \ref{lem:2}}.
We used the Magma Computational Algebra System to justify the statement.
In $\mathbb{Q}(\alpha)$ the coefficient
$$
a=\frac{1}{(\alpha-\beta)(\alpha-\gamma)}
$$
can be written as
$$
a=\frac{\alpha}{\alpha^2+2\alpha+3},
$$
while
$$
\alpha a=\frac{\alpha^2}{\alpha^2+2\alpha+3}.
$$

The splitting field $K=\mathbb{Q}(\alpha,\beta)$ is generated by a zero of the symmetric polynomial
\begin{equation}\label{split}
p(x)=x^6-x^5+2x^4-3x^3+2x^2-x+1.
\end{equation}
In order to determine the zeros of the above polynomial, it is easy to see, that
$$
x+\frac{1}{x}=\tau,
$$
where $\tau$ is $\alpha$ or $\beta$ or $\gamma$. Thus, denoting by $\ve$ a zero of (\ref{split}), we obtain
$$
\ve=\frac{\alpha+\sqrt{\alpha^2-4}}{2}.
$$
We have the representations
$$
\alpha=-\ve^5+\ve^4-2\ve^3+3\ve^2-\ve+1,
$$
further
$$
a=\frac{-5\ve^5+9\ve^4-14\ve^3+19\ve^2-13\ve+6}{22},
$$
and
$$
\alpha a=\frac{-2\ve^5-3\ve^4+\ve^3+\ve^2+8\ve-2}{22}.
$$
Finally, by testing $a$ and $\alpha a$ by the {\tt SquareRoot} function of Magma, we can conclude that neither $a$ nor $\alpha a$ is a square in $\mathbb{Q}(\alpha,\beta)$.
\qed\medskip

\noindent Let us continue with Lemma \ref{lem:1}.

\medskip

\noindent {\it Proof of Lemma \ref{lem:1}}.
Under the assumption that (\ref{eq:triples}) has infinitely many solutions, we have already deduced that there are finite sets $I_v,I_w$ and coefficients $e_v,e_w$ and $e_{v,i},e_{w,j}$ for $i\in I_v,j\in I_w$ in $K$ such that
\begin{equation*}\begin{split}
v & =  \alpha^{(x-y+z-\ka)/2} \left(e_v+\sum_{i\in I_v} e_{v,i} \alpha^{A_{v,i}({\bf x})} \beta^{B_{v,i}({\bf x})} \gamma^{C_{v,i}({\bf x})}\right),\\
w & =  \alpha^{(-x+y+z-\ka)/2} \left(e_w+\sum_{i\in I_w} e_{w,i} \alpha^{A_{w,i}({\bf x})} \beta^{B_{w,i}({\bf x})} \gamma^{C_{w,i}({\bf x})}\right).
\end{split}\end{equation*}
Since $1+vw=T_z=a\alpha^z+b\beta^z+c\gamma^z$, we get
\begin{equation}\label{eq:S}a\alpha^z+b\beta^z+c\gamma^z-\alpha^zS=1,\end{equation}
where $S$ is a certain finite sum of monomials.

This is an $S$-unit equation with certain coefficients and of a certain length where $S$ is the multiplicative group generated by $\{\alpha,\beta,-1\}$ inside the complex numbers (note that
$\gamma=-(\alpha \beta)^{-1}$). Replacing $\gamma=-\beta^{-1} \alpha^{-1}$, equation \eqref{eq:S} is of the form
\begin{equation}
\label{eq:non}
\sum_{i\in I} e_i \alpha^{K_i({\bf x})} \beta^{L_i({\bf x})}=0,
\end{equation}
where $I$ is some finite set, $e_i$ are nonzero coefficients and $K_i({\bf x})$ and $L_i({\bf x})$ are linear forms in ${\bf x}$ with integer coefficients. It is assumed of course that if $i\ne j$ then
$(K_i({\bf x}),L_i({\bf x}))\ne (K_j({\bf x}),L_j({\bf x}))$. Assume that $i\ne j$ are in the same non-degenerate component of a sub-equation of \eqref{eq:non}. Then
$\alpha^{L_i({\bf x})-L_j({\bf x})}\beta^{M_i({\bf x})-M_j({\bf x})}$ belongs to some finite list of numbers. Since $\alpha$ and $\beta$ are multiplicatively independent, if follows that $(K_i-K_j)({\bf x})=0$ and
$(L_i-L_j)({\bf x})=0$. Since at least one of these two forms is nonzero, it follows that we may assume that $L({\bf x})=0$ for some non-zero form $L$. Thus, there exist $r_i,s_i,t_i$ for $i=1,2,3$ such that
$$
x=r_1 p+s_1 q+ t_1,\qquad y=r_2 p+s_2 q+t_2,\qquad z=r_3 p+s_3 q+t_3.
$$
Let $\Delta$ be the least common multiple of the denominators of $r_i,~s_i$ for $i=1,2,3$ and let $p_0,~q_0$ be such that for infinitely many pairs $(p,q)$ we have $p\equiv p_0\pmod \Delta$ and $q\equiv q_0\pmod \Delta$.
Then $p=p_0+\Delta \lambda,~q=q_0+\Delta \mu$, and
\begin{equation*}\begin{split}
x & =  (r_1\Delta) \lambda+(s_1\Delta \mu)+(r_1p_0+s_1q_0+t_1)\\
y & =  (r_2\Delta) \lambda+(s_2\Delta \mu)+(r_2p_0+s_2q_0+t_2)\\
z & =  (r_3\Delta) \lambda+(s_3\Delta \mu)+(r_3p_0+s_3q_0+t_3).
\end{split}\end{equation*}
Since $r_i\Delta,~s_i\Delta$, $x,~y,~z$ are integers, we conclude that $r_i p_0+s_i q_0+t_i$ are also integers  for $i=1,2,3$. Thus, we may assume that in fact $r_i,s_i,t_i$ are all integers.
A similar argument now shows that our equation is of the form
$$
\sum_{i\in J} f_i \alpha^{M_i({\bf r})} \beta^{N_i({\bf r})}=0,
$$
where ${\bf r}=(\lambda,\mu)$, $J$ is some finite set of indices and $f_i$ are nonzero coefficients for $i\in J$. Again we may assume that $(M_i({\bf r}),N_i({\bf r}))\ne (M_j({\bf r}),N_j({\bf r}))$ for $i\ne j$ in $J$.  Applying again
the theorem on non-degenerate solutions to $S$-unit equations, we get that there exists some finite set of numbers $\Gamma$ such that for some $i\ne j$, the quantity
$$
\alpha^{(M_i-M_j)({\bf r})} \beta^{(N_i-N_j){\bf r}}\in \Gamma.
$$
This shows that each such ${\bf r}$ lies on a finite collection of lines, and since there are infinitely many ${\bf r}$ to start with,
we get that infinitely many of our ${\bf r}$ are on the same line.
This finishes the proof of the lemma.
\qed


\begin{thebibliography}{9999}

\bibitem{BL} BRAVO, J. J.---LUCA, F.: \textit{On a conjecture about repdigits in k-generalized Fibonacci sequences}, Publ.~Math.~Debrecen \textbf{82} (2013),  623--639.

\bibitem{FLSz} FUCHS, C.---LUCA, F.---SZALAY, L.: \textit{Diophantine triples with values in binary recurrences}, Ann.~Sc.~ Norm.~Super.~Pisa~Cl.~Sc.~(5) \textbf{7} (2008), 579--608.

\bibitem{duje0} DUJELLA, A.: \textit{There are only finitely many Diophantine quintuples}. J. Reine Angew. Math. \textbf{566} (2004), 183--214.

\bibitem{duje} DUJELLA, A.: \textit{Diophantine $m$-tuples}, web page available at: {\sf https://web.math.pmf.unizg.hr/~duje/dtuples.html}.

\bibitem{ESS} EVERTSE, J.-H.---SCHMIDT, W. M.---SCHLICKEWEI, H.-P.: \textit{Linear equations in variables which lie in a multipilicative group}, Ann. of Math. (2) \textbf{155} (2002), 807--836.

\bibitem{Ev}  EVERTSE, J.-H.: \textit{An improvement of the quantitative Subspace Theorem}, Compos. Math. \textbf{101} (1996), 225--311.

\bibitem{Fu2} FUCHS, C.: \textit{Polynomial-exponential equations and linear recurrences}, Glas. Mat. \textbf{38(58)} (2003), no. 2, 233--252.

\bibitem{Fu5} FUCHS, C.: \textit{Diophantine problems with linear recurrences via the Subspace Theorem}, Integers {\bf 5} (2005), no. 3, A8.

\bibitem{Fu4} FUCHS, C.: \textit{Polynomial-exponential equations involving multi-recurrences}, Studia Sci. Math. Hungar. {\bf 46} (2009), 377--398.

\bibitem{ft} FUCHS, C.---TICHY, R. F.: \textit{Perfect powers in linear recurrence sequences}, Acta~Arith. \textbf{107.1} (2003), 9--25.

\bibitem{LSz1} LUCA, F.---SZALAY, L.: \textit{Fibonacci Diophantine Triples}, Glas. Mat. \textbf{43(63)} (2008), 253--264.

\bibitem{LSz2} LUCA, F.---SZALAY, L.: \textit{Lucas Diophantine Triples}, Integers \textbf{9} (2009), 441--457.

\bibitem{ISz1} IRMAK, N.---SZALAY, L.: \textit{Diophantine triples and reduced quadruples with the Lucas sequence of recurrence $u_n=Au_{n-1}-u_{n-2}$}, Glas.~Mat.~\textbf{49} (2014), 303--312.

\bibitem{RL} GOMEZ RUIZ, C. A.---LUCA, F.: \textit{Tribonacci Diophantine quadruples}, Glas.~Mat.~ \textbf{50} (2015), no. 1, 17--24.

\bibitem{RL1} GOMEZ RUIZ, C. A.---LUCA, F.: \textit{Diophantine quadruples in the sequence of shifted Tribonacci numbers}, Publ.~Math. Debrecen~\textbf{86} (2015), no. 3-4, 473--491.

\bibitem{S} SPICKERMAN, W. R.: \textit{Binet's formula for the Tribonacci numbers}, Fibonacci Q. \textbf{20} (1982), 118--120.

\end{thebibliography}
\end{document}